\documentclass[12pt]{amsart}       
\usepackage{txfonts}
\usepackage{extarrows}
\usepackage{amssymb}
\usepackage{eucal}
\usepackage{bbm}
\usepackage{graphicx}
\usepackage{amsmath}
\usepackage{amscd}
\usepackage[all]{xy}           
\usepackage{amsfonts,latexsym}
\usepackage{xspace}
\usepackage{epsfig}
\usepackage{float}
\usepackage{color}
\usepackage{shuffle}
\usepackage{fancybox}
\usepackage{colordvi}
\usepackage{multicol}
\usepackage{colordvi}
\usepackage{mathrsfs}
\usepackage{ifpdf}
\ifpdf
  \usepackage[colorlinks,final,backref=page,hyperindex]{hyperref}
\else
  \usepackage[colorlinks,final,backref=page,hyperindex,hypertex]{hyperref}
\fi
\usepackage[active]{srcltx} 
\usepackage{tikz}
\usepackage{graphicx}
\usepackage{enumitem}
\usepackage{amsmath,amssymb}
\usepackage{xcolor}
\usetikzlibrary{arrows.meta,positioning,calc,fit}

\DeclareMathAlphabet{\cmcal}{OMS}{cmsy}{m}{n}
\newcommand{\Part}{\cmcal{P}}

\newtheorem{theorem}{Theorem}[section]
\newtheorem{proposition}[theorem]{Proposition}
\newtheorem{lemma}[theorem]{Lemma}

\newtheorem{prop-def}{Proposition-Definition}[section]
\newtheorem{coro-def}{Corollary-Definition}[section]

\newtheorem{conjecture}[theorem]{Conjecture}

\theoremstyle{definition}
\newtheorem{definition}[theorem]{Definition}
\newtheorem{remark}[theorem]{Remark}

\newcommand{\nc}{\newcommand}
\nc{\tred}[1]{\textcolor{red}{#1}}
\nc{\tblue}[1]{\textcolor{blue}{#1}}
\nc{\tgreen}[1]{\textcolor{green}{#1}}
\nc{\tpurple}[1]{\textcolor{purple}{#1}}
\nc{\btred}[1]{\textcolor{red}{\bf #1}}
\nc{\btblue}[1]{\textcolor{blue}{\bf #1}}
\nc{\btgreen}[1]{\textcolor{green}{\bf #1}}
\nc{\btpurple}[1]{\textcolor{purple}{\bf #1}}
\nc{\NN}{{\mathbb N}}
\nc{\ncsha}{{\mbox{\cyr X}^{\mathrm NC}}} \nc{\ncshao}{{\mbox{\cyr
X}^{\mathrm NC}_0}}

\newcommand{\delete}[1]{}

\nc{\mlabel}[1]{\label{#1}}
\nc{\mcite}[1]{\cite{#1}}
\nc{\mref}[1]{\ref{#1}}
\nc{\meqref}[1]{\eqref{#1}}
\nc{\mbibitem}[1]{\bibitem{#1}}

\delete{
\nc{\mlabel}[1]{\label{#1}{\hfill \hspace{1cm}{\bf{{\ }\hfill(#1)}}}}
\nc{\mcite}[1]{\cite{#1}{{\bf{{\ }(#1)}}}}
\nc{\mref}[1]{\ref{#1}{{\bf{{\ }(#1)}}}}
\nc{\meqref}[1]{\eqref{#1}{{\bf{{\ }(#1)}}}}
\nc{\mbibitem}[1]{\bibitem[\bf #1]{#1}}
}
\nc{\sha}{{\mbox{\cyr X}}}  
\newfont{\scyr}{wncyr10 scaled 550}
\nc{\ssha}{\mbox{\bf \scyr X}}
\nc{\shap}{{\mbox{\cyrs X}}} 
\nc{\shpr}{\diamond}    
\nc{\shp}{\ast} \nc{\shplus}{\shpr^+}
\nc{\shprc}{\shpr_c}    
\nc{\dep}{\mrm{dep}} \nc{\lc}{\lfloor} \nc{\rc}{\rfloor}
\nc{\db}{\leq_{\rm db}} \nc{\bfk}{{\bf k}}

\nc{\cala}{{\mathcal A}} \nc{\calb}{{\mathcal B}}
\nc{\calc}{{\mathcal C}}
\nc{\cald}{{\mathcal D}} \nc{\cale}{{\mathcal E}}
\nc{\calf}{{\mathcal F}} \nc{\calg}{{\mathcal G}}
\nc{\calh}{{\mathcal H}} \nc{\cali}{{\mathcal I}}
\nc{\call}{{\mathcal L}} \nc{\calm}{{\mathcal M}}
\nc{\caln}{{\mathcal N}} \nc{\calo}{{\mathcal O}}
\nc{\calp}{{\mathcal P}} \nc{\calr}{{\mathcal R}}
\nc{\cals}{{\mathcal S}} \nc{\calt}{{\mathcal T}}
\nc{\calu}{{\mathcal U}} \nc{\calw}{{\mathcal W}} \nc{\calk}{{\mathcal K}}
\nc{\calx}{{\mathcal X}} \nc{\CA}{\mathcal{A}}

\nc{\fraka}{{\mathfrak a}} \nc{\frakA}{{\mathfrak A}}
\nc{\frakb}{{\mathfrak b}} \nc{\frakB}{{\mathfrak B}}
\nc{\frakc}{{\mathfrak c}}
\nc{\frakD}{{\mathfrak D}} \nc{\frakF}{\mathfrak{F}}
\nc{\frakf}{{\mathfrak f}} \nc{\frakg}{{\mathfrak g}}
\nc{\frakH}{{\mathfrak H}} \nc{\frakL}{{\mathfrak L}}
\nc{\frakM}{{\mathfrak M}} \nc{\bfrakM}{\overline{\frakM}}
\nc{\frakm}{{\mathfrak m}} \nc{\frakP}{{\mathfrak P}}
\nc{\frakN}{{\mathfrak N}} \nc{\frakp}{{\mathfrak p}}
\nc{\frakS}{{\mathfrak S}} \nc{\frakT}{\mathfrak{T}}
\nc{\frakX}{{\mathfrak X}}

\font\cyr=wncyr10 \font\cyrs=wncyr7
\nc{\li}[1]{\textcolor{blue}{Nan:#1}}
\nc{\lir}[1]{\textcolor{red}{Li:#1}}
\nc{\yi}[1]{\textcolor{blue}{Yi: #1}}
\nc{\xing}[1]{\textcolor{blue}{Xing:#1}}
\nc{\revise}[1]{\textcolor{red}{#1}}
\nc{\nan}[1]{\textcolor{blue}{Nan:#1}}

\numberwithin{equation}{section}
\newcommand{\R}{\mathbb R}
\newcommand{\C}{\mathbb C}
\newcommand{\K}{\mathbb K}

\newcommand{\End}{\operatorname{End}}
\newcommand{\Id}{\operatorname{Id}}
\newcommand{\Ree}{\operatorname{Re}}
\newcommand{\spec}{\operatorname{spec}}

\newcommand{\norm}[1]{\lVert #1\rVert}
\newcommand{\ot}{\otimes}
\nc{\X}{\mathbf{X}}
\nc{\x}{\mathbb{X}}
\nc{\G}{\Gamma}

\begin{document}

\title[Exact Signature Tail Asymptotics for Pure Rough Paths]{Exact Signature Tail Asymptotics for Pure Rough Paths}
%
%
\author{Nannan Li}
\address{School of Mathematics and Statistics, Lanzhou University
Lanzhou, 730000, China
}
\email{linn2024@lzu.edu.cn}

\author{Xing Gao$^{*}$}\thanks{*Corresponding author}
\address{School of Mathematics and Statistics, Lanzhou University
Lanzhou, 730000, China; Gansu Provincial Research Center for Basic Disciplines of Mathematics
and Statistics, Lanzhou, 730070, China
}
\email{gaoxing@lzu.edu.cn}
\begin{abstract}
We prove~\cite[Conjecture 2.12]{BGS20} on the signature tail asymptotics of pure rough paths and extend it to arbitrary reasonable tensor norms. In more details, let
\[
\mathbf X_t=\exp(tl) \,\text{ with }\, l=l_1+\cdots+l_m\,\text{ and }\, l_r\in\mathcal L_r(V),
\]
be a pure $m$-rough path over a finite dimensional real or complex Banach space, and equip the tensor powers of $V$ with arbitrary reasonable tensor algebra norms. We prove that
\[
\limsup_{n\to\infty}\left(\left(\frac{n}{m}\right)!\left\|\pi_n(\exp l)\right\|_n\right)^{m/n}=\|l_m\|_m .
\]
In particular, this identifies the signature tail with the local $m$-variation of the pure rough path.  The upper bound was obtained in~\cite{BGS20}; the main contribution of the paper is the matching lower bound.  Its proof is based on finite dimensional developments and a norming cyclic construction. For every top-level tensor $l_m$, we also build a contractive development in which $\|l_m\|_m$ appears as an eigenvalue at degree $m$.
\end{abstract}

\makeatletter
\@namedef{subjclassname@2020}{\textup{2020} Mathematics Subject Classification}
\makeatother
\subjclass[2020]{
60L20, 
60L10, 
}

\keywords{rough path, signature, reasonable tensor algebra norm. }

\maketitle

\tableofcontents

\setcounter{section}{0}

\allowdisplaybreaks

\section{Introduction}\label{sec:introduction}
Controlled differential equations driven by irregular signals arise naturally in stochastic analysis and related applications.  A fundamental example is provided by stochastic differential equations driven by Brownian motion. More generally, one is led to equations of the form
\begin{equation*}
dY_t=\sum_{i=1}^d V_i(Y_t)\,dX_t^i,
\end{equation*}
where $V_i:\mathbb R^N\to\mathbb R^N$, $X:[0,T]\to\mathbb R^d$, and $Y:[0,T]\to\mathbb R^N$. Rough path theory, initiated by Lyons~\cite{Ly98} and further developed in several directions~\cite{FH20, FV10, Gu04, Gu10, LQ02}, gives a deterministic and analytically robust framework for such equations by enhancing an irregular path with its iterated-integral data.  In this way, it extends the classical Lebesgue--Stieltjes theory for bounded-variation paths continuously with respect to the rough path topology.

The resulting collection of iterated integrals is encoded by the signature, originating from Chen's theory of iterated integrals~\cite{Ch73}. The signature is a fundamental object in rough path theory.  For a path of bounded variation $X:[0,T]\to V$, it is the tensor series
\[
S(X)=1+\sum_{n=1}^{\infty}\int_{0<t_1<\cdots<t_n<T}dX_{t_1}\otimes\cdots\otimes dX_{t_n}.
\]
Lyons' extension theorem extends this construction to rough paths by canonically determining all higher tensor levels from the truncated rough path data; see Theorem~\ref{thm:1} and~\cite{Ly98}.  The significance of the signature is based on the fundamental uniqueness theorem: every weakly geometric rough path is determined by its signature up to tree-like equivalence~\cite{BGLY16, HL10}.  This result shows that the signature encodes essentially all information about the driving signal.  However, the uniqueness theorem is non-constructive and does not by itself explain how one can reconstruct a path, or recover quantitative information about it, from its signature.  This has motivated a line of work on signature inversion, reconstruction, and quantitative signature asymptotics~\cite{BG23, BGS20, Ch18, CLN18, Ge17, LX15, LX18}.

The present paper concerns a quantitative form of this principle.  Instead of asking whether the full signature determines the path, we ask what geometric information can be recovered from the asymptotic behaviour of its high tensor levels.  This is closely related to the following question raised by Boedihardjo, Geng and Souris.\\

\noindent\textbf{Question~\cite{BGS20}.}
Are there explicit and elegant formulae allowing us to recover intrinsic properties of the path from its signature tail asymptotics?\\

For bounded-variation paths, the estimate
\[
\left\|\int_{0<t_1<\cdots<t_n<T}dX_{t_1}\otimes\cdots\otimes dX_{t_n}\right\|\leq\frac{\|X\|_{1\textnormal{-var}}^n}{n!}
\]
suggests normalising the $n$-th signature level by $n!$. This leads to the Hambly--Lyons length recovery problem. For tree-reduced bounded variation paths, one expects the path length to be recovered from the asymptotics of the normalised signature levels.  This problem was implicit in~\cite{HL10} and was later made explicit by Chang, Lyons and Ni~\cite{CLN18}. More recently, Boedihardjo and Geng proved the corresponding formula for a class of planar bounded variation paths by using $\mathrm{SL}_2(\mathbb R)$-developments and angle dynamics~\cite{BG23}.

For a $p$-rough path, the natural factorial scale is instead $(n/p)!=\Gamma(n/p+1)$. This leads to the signature tail functional
\[
L_p(\mathbf X):=\limsup_{n\to\infty}\left(\left(\frac{n}{p}\right)!\left\|X^n_{0,T}\right\|_n\right)^{p/n},
\]
introduced below in \eqref{eq:3}.  The first indication of the geometric meaning of $L_p(\mathbf X)$ was provided by Boedihardjo and Geng~\cite{BG19}.  They proved that, when $\mathbf X$ is the Brownian rough path and $p=2$, the corresponding signature tail asymptotic is a deterministic constant multiple of the quadratic variation.  This suggested that $L_p(\mathbf X)$ should be closely related to a suitable local $p$-variation of the rough path. The use of a limsup is important, since some signature levels may vanish identically, for example when the logarithm of the signature is homogeneous.

Boedihardjo, Geng and Souris~\cite{BGS20} studied this problem for pure rough paths. For a pure rough path $\mathbf X$,
\[
\mathbf X_t=\exp(tl)\in G^{(m)}(V),
\]
they proved the sharp upper estimate
\[
L_m(\mathbf X)\le \|\pi_m(l)\|_m
\]
and obtained uniform lower estimates, together with exact identities in several special cases. Here $\pi_m$ is the canonical projection.  They conjectured that the upper bound should always be attained, and formulated this as follows.

\begin{conjecture}\cite[Conjecture 2.12]{BGS20}\label{conj:1}
For every pure $m$-rough path $\mathbf X_t=\exp(tl)\in G^{(m)}(V)$, the tail asymptotic quantity $L_m(\mathbf X)$ of its signature equals the local $m$-variation of $\mathbf X$.  In view of Proposition~2.10 of~\cite{BGS20}, that is,
\[
L_m(\mathbf X)=\|\pi_m(l)\|_m .
\]
\end{conjecture}

Thus the remaining difficulty is not the upper estimate, which follows from factorial-decay bounds, but the matching lower estimate. Our main result proves Conjecture~\ref{conj:1}, and does so for arbitrary reasonable tensor algebra norms.  More precisely, if $V$ is a finite-dimensional Banach space over $\R$ or $\C$ and $\mathbf X_t=\exp(tl)$ is a pure $m$-rough path, then
\begin{equation*}
L_m(\mathbf X)
=
\|\pi_m(l)\|_m .
\end{equation*}

The paper is organized as follows.  Section~\ref{sec:2} recalls reasonable tensor algebra norms, signatures of rough paths, the tail functional, pure rough paths, and the known upper estimate.  Section~\ref{sec:lower-bound} proves the abstract lower-bound principle using finite dimensional developments and dilation (Proposition~\ref{prop:2}).  Section~\ref{sec:cyclic} constructs the norming cyclic development and completes the proof of the main theorem (Theorem~\ref{thm:3}).

\section{Reasonable tensor norms and known upper estimate}\label{sec:2}
In this section, we introduce the tensor-norm assumptions, the basic notation for pure rough paths, and several elementary estimates needed in the sequel.

Let $V$ be a finite dimensional Banach space over $\K$, where $\K=\R$ or $\C$. We write $V^{\otimes n}$ for the algebraic tensor power and assume that each
$V^{\otimes n}$ is equipped with a norm $\norm{\cdot}_n$.  We put
$V^{\otimes0}=\K$ and $\norm{c}_0=|c|$.

\begin{definition}\label{def:1}
A sequence $\{\norm{\cdot}_n:n\ge1\}$ is called a \emph{sequence of reasonable tensor algebra norms} if
\begin{enumerate}[label=(\roman*), itemsep=0.35em, topsep=0.35em]
\item $\norm{\cdot}_1$ is the given norm on $V$;

\item $\norm{\xi\ot\eta}_{m+n}\le \norm{\xi}_m\norm{\eta}_n$ for all $\xi\in V^{\otimes m}$ and $\eta\in V^{\otimes n}$;

\item $\norm{\varphi\ot\psi}_{(V^{\otimes(m+n)})^*}\le \norm{\varphi}_{(V^{\otimes m})^*}\norm{\psi}_{(V^{\otimes n})^*}$ for all $\varphi\in(V^{\otimes m})^*$ and $\psi\in(V^{\otimes n})^*$;

\item every permutation operator on $V^{\otimes n}$ is an isometry.
\end{enumerate}
\end{definition}

Let $T((V))$ be the infinite tensor algebra consisting of tensor series $\xi=(\xi_0, \xi_1, \ldots)$ with $\xi_m\in V^{\otimes m}$ for each $m$. Given $m\ge 1$, let
\[
T^{(m)}(V):= \bigoplus_{k=0}^{m} V^{\otimes k}
\]
be the \emph{truncated tensor algebra} of degree $n$. For $l\in T^{(m)}(V)$, write $l=l_1+\cdots+l_m$ with $l_r\in V^{\otimes r}$. We next fix the rough-path convention used in the rest of the paper. Denote by  $$\Delta_T=\{(s,t):0\le s\le t\le T\}.$$

\begin{definition}\label{def:2}
\begin{enumerate}
\item Let $m\in \mathbb{Z}_{\geq 1}$. A \emph{multiplicative functional of degree $m$}  is a continuous map
\[
\mathbf X:\Delta_T\longrightarrow T^{(m)}(V),\quad (s,t)\longmapsto \mathbf X_{s,t}=(1,X^1_{s,t},\ldots,X^m_{s,t}),
\]
such that the Chen relation
\[
\mathbf X_{s,u}=\mathbf X_{s,t}\otimes \mathbf X_{t,u}
\]
holds  for all $0\le s\le t\le u\le T$. 

\item  Let $p\in \mathbb{R}_{\geq 1}$ and  $m:=\lfloor p\rfloor$. A multiplicative functional of degree $m$ is called a \emph{$p$-rough path} if
\[
\|\mathbf X\|_{p\textnormal{-var};[0,T]}:=\sum_{k=1}^{m}\Bigg(\sup_{\Part}\sum_{[u,v]\in\Part}\|X^k_{u,v}\|_k^{p/k}\Bigg)^{k/p}<\infty,
\]
where the supremum is taken over all finite partitions $\Part$ of $[0,T]$.
\end{enumerate}
\end{definition}

Indeed, once $\mathbf X_{0,t}$ is given, Chen's relation determines all increments by
\[
\mathbf X_{s,t}=\mathbf X_{0,s}^{-1}\otimes \mathbf X_{0,t},
\]
where the inverse is taken in the truncated tensor algebra. Conversely, the family of increments clearly determines the path $t\mapsto \mathbf X_{0,t}$. We shall use these two descriptions interchangeably.

Recall that there is a natural Lie structure on the tensor algebra given by
\[
[\xi,\eta]:= \xi\otimes \eta-\eta\otimes \xi .
\]
The space of homogeneous Lie polynomials of degree $n$, denoted as $\mathcal{L}_n(V)$, is the norm completion of the algebraic space $\mathcal{L}_n^a(V)$ defined inductively by
\[
\mathcal{L}_1^a(V):=V, \quad \mathcal{L}_{n+1}^a(V):= [V,\mathcal{L}_n^a(V)].
\]
For $m\in \mathbb{Z}_{\geq 1}$, set
\[
\mathcal L^{(m)}(V):=\bigoplus_{k=1}^m \mathcal L_k(V).
\]
This is the step-$m$ free nilpotent Lie algebra over $V$, viewed as a subspace of $T^{(m)}(V)$. The corresponding step-$m$ free nilpotent group is
\[
G^{(m)}(V)=\exp\bigl(\mathcal L^{(m)}(V)\bigr)\subset T^{(m)}(V),
\]
where the exponential is understood in the truncated tensor algebra.

\begin{definition}\label{def:3}
A $p$-rough path is said to be \emph{weakly geometric} if it takes values in the group $G^{(\lfloor p\rfloor)}(V)$.
\end{definition}

An important aspect of rough path theory is the characterization of rough paths in terms of the so-called path signature, which is a generalized notion of iterated path
integrals. For a truncated rough path, the existence and uniqueness of the higher-order signature levels is guaranteed by Lyons' extension theorem~\cite{Ly98}, which we recall in the form needed below.

\begin{theorem} \label{thm:1}
Let $\mathbf{X}=(\mathbf{X}_{s,t})_{0\le s\le t\le T}$ be a $p$-rough path. Then there exists a unique extension of \, $\X$ to a multiplicative functional
\[
\mathbb{X}:\Delta_T\longrightarrow T((V)), \quad (s,t)\mapsto \mathbb{X}_{s,t}=\bigl(1,X^1_{s,t},\ldots,X^{\lfloor p\rfloor}_{s,t}, \ldots, X^n_{s,t}, \ldots\bigr),
\]
whose restriction to $T^{(n)}(V)$ has finite total $p$-variation for all $n\ge \lfloor p\rfloor+1$. Moreover, there exist a universal constant $\beta_p$ depending only on $p$ and a non-negative function $\omega_{\mathbf{X}}(s,t)$ related to the $p$-variation of $\mathbf{X}$, such that
\begin{equation}\label{eq:2}
\bigl\|X^n_{s,t}\bigr\| \le \frac{\omega_{\mathbf{X}}(s,t)^{n/p}}{\beta_p (n/p)!}, \quad \textnormal{for all $n\ge 1$ and $(s,t)\in\Delta_T$},
\end{equation}
where the factorial $(n/p)!$ is defined by using the Gamma function.
\end{theorem}

\begin{definition}\label{def:4}
The tensor series $\mathbb{X}_{0,T}\in T((V))$ is called the \emph{signature} of $\mathbf{X}$. It is usually denoted as $S(\mathbf{X})$.
\end{definition}

We recall the signature-tail functional associated with a weakly geometric $p$-rough path $\mathbf X$.  The factorial decay estimate~\eqref{eq:2} shows that the natural scale at level $n$ is governed by $\left(\frac{n}{p}\right)!$.  We therefore define
\begin{equation}\label{eq:3}
L_p(\mathbf X):=\limsup_{n\to\infty}\left(\left(\frac{n}{p}\right)!\bigl\|X^n_{0,T}\bigr\|\right)^{p/n}.
\end{equation}
The use of a limsup is essential, since individual signature levels may vanish along infinitely many degrees. In the sequel we focus on pure rough paths of degree $m$, namely rough paths of the form
\[
\mathbf X_t=\exp(tl),\quad 0\le t\le 1,
\]
where $l=l_1+\cdots+l_m\in\mathcal L^{(m)}(V)$ and $l_r\in\mathcal L_r(V)$.

We now recall the definition of pure rough paths and record the elementary properties needed later.

\begin{definition}\cite[Definition 2.8]{BGS20}\label{def:5}
Let $m\in \mathbb{Z}_{\geq 1}$.
A \emph{pure $m$-rough path} is a weakly geometric rough path of the form
\[
\mathbf{X}_t=\exp(tl)\in G^{(m)}(V), \quad 0\le t\le 1,
\]
where $l\in\mathcal{L}^{(m)}(V)$.
\end{definition}

\begin{lemma}\cite[Proposition 2.10]{BGS20}\label{lem:1}
Let $m\ge 1$. A pure $m$-rough path $\mathbf{X}_t=\exp(tl)$ is a rough path with roughness $m$ in the sense of Definition~\ref{def:2}. In addition, the local $m$-variation of \,$\X$ coincides with the norm of the highest degree component of $l$, in the sense that
\[
\lim_{n\to\infty}\sum_{k=1}^{m}\Bigg(\sum_{t_i\in\Part_n}\bigl\|X^k_{t_{i-1},t_i}\bigr\|^{m/k}\Bigg)^{k/m}=\bigl\|\pi_m(l)\bigr\|_m
\]
for any sequence of finite partitions $\Part_n$ over $[0,1]$ whose mesh size tends to zero, where
\[
\pi_m:T^{(m)}(V)\longrightarrow V^{\otimes m}
\]
is the canonical projection.
\end{lemma}

\begin{lemma}\cite[Proposition 2.11]{BGS20}\label{lem:2}
Let $\mathbf{X}_t=\exp(tl)$ be a pure $m$-rough path. Then its signature is equal to $\exp(l)$, where the exponential is now taken over the infinite tensor algebra $T((V))$. In addition, up to tree-like equivalence, this is the only weakly geometric rough path whose signature is $\exp(l)$.
\end{lemma}

The following upper estimate is due to Boedihardjo--Geng--Souris.  Its proof uses only the preceding reasonable tensor norm assumptions and does not require the projective tensor norm.

\begin{lemma}\cite[Theorem 4.1]{BGS20}\label{prop:1}
Let $l\in T^{(m)}(V)$ and  $\X_t=\exp(tl)$.  Then
\[
  L_m(\mathbf{X})\le \norm{\pi_m(l)}_m.
\]
\end{lemma}

It remains to prove the reverse inequality.  The main point is that, for a general reasonable tensor norm, one cannot rely on the universal property of the projective tensor product.  Instead, we construct a finite-dimensional development explicitly and prove directly that the resulting tensor-level maps are bounded on all degrees.

\section{A tensor-level lower bound from developments}\label{sec:lower-bound}
This section establishes the main estimate needed for the proof and uses it to prove the main theorem.
Given a Banach space $W$, we use the notation $\End(W)$ (resp. $\operatorname{Aut}(W)$) to denote the space of continuous linear endomorphisms (resp. automorphisms) over $W$, equipped with the operator norm.

\begin{definition}\cite[Definition 4.5]{BGS20}\label{def:6}
Let $V$ be a  Banach space over $\K$.
\begin{enumerate} 
\item A \emph{Lie algebraic development} $\Phi$ of $V$ consists of a linear map $F:V\longrightarrow \mathfrak{g}$ into a Lie algebra $\mathfrak{g}$ over $\K$ and a representation $\rho:\mathfrak{g}\longrightarrow \operatorname{End}(W)$ of $\mathfrak{g}$ on a  Banach space $W$ over $\K$ such that
\[
\Phi=\rho\circ F
\]
is continuous. 

\item The Lie algebraic development $\Phi$ is said to be \emph{finite dimensional} if $\mathfrak{g}$ and $W$ are both finite dimensional. 
    
\item In situations when the intermediate Lie algebra $\mathfrak{g}$ is not relevant, we simply refer to $\Phi:V\to \operatorname{End}(W)$ as a \emph{development}.
\end{enumerate}
\end{definition}

Let $W$ be a finite dimensional Banach space over $\K$ and let $\Phi:V\to\End(W)$ be linear.  By the universal property of tensor product, for every $n\ge1$, $\Phi$  induces a continuous linear map
\[
\Phi^{(n)}:V^{\otimes n}\longrightarrow \End(W)
\]
such that
\[
\Phi^{(n)}(v_1\otimes\cdots\otimes v_n)
=
\Phi(v_1)\cdots\Phi(v_n),
\]
and we assume that
\begin{equation}\label{eq:4}
\norm{\Phi^{(n)}}_{V^{\otimes n}\to\End(W)}\le C^n, \quad n\ge1,
\end{equation}
for some $C>0$.  We then extend $\Phi$ multiplicatively to tensor series by
\begin{equation}\label{eq:4-1}
\Phi\bigl((\xi_0,\xi_1,\xi_2,\ldots)\bigr):=\xi_0\Id+\sum_{n\ge1}\Phi^{(n)}(\xi_n)
\end{equation}
whenever the series converges.

We now recall the relevant statement from~\cite{BGS20}. Under the given development $\Phi$, every rough path $(\mathbf{X}_t)_{0\le t\le T}$ over $V$ can be developed onto the automorphism group $\operatorname{Aut}(W)$ by solving the linear ODE
\begin{equation*}
\begin{cases}
d\Gamma_t=\Gamma_t\cdot \Phi(d\X_t), \quad 0\le t\le T,\\
\Gamma_0=\operatorname{Id}.
\end{cases}
\end{equation*}
Using Picard's iteration, it is easily seen that
\begin{align*}
\Gamma_t&\ =\sum_{n=0}^{\infty}\int_{0<t_1<\cdots<t_n<t}\Phi(d\X_{t_1})\cdots \Phi(d\X_{t_n}) \\
&\ =\sum_{n=0}^{\infty}\Phi^{(n)}\left(\int_{0<t_1<\cdots<t_n<t}d\X_{t_1}\otimes\cdots\otimes d\X_{t_n}\right) \\
&\ =\Phi(\x_{0,t}),
\end{align*}
where $\x_{0,t}$ is the Lyons extension of $\mathbf{X}$ given by Theorem~\ref{thm:1}. Note that by the factorial decay estimate in the same theorem, $\Phi(\x_{0,t})$ is well defined. In particular, we have
\[
\Gamma_T=\Phi(S(\mathbf{X})).
\]

We shall use the following elementary estimate to control the signature expansion of a development.

\begin{lemma}\label{lem:3}
For every $m\in \mathbb{Z}_{\geq 1}$, there is a constant $K_m>0$ such that
\begin{equation}\label{eq:10+1}
\sum_{n=0}^{\infty}\frac{x^n}{(n/m)!}\le K_m(1+x)^m\exp(x^m), \quad x\ge0.
\end{equation}
Consequently,
\[
\limsup_{x\to\infty}x^{-m}\log\Bigg(\sum_{n=0}^{\infty}\frac{x^n}{(n/m)!}\Bigg)\le1.
\]
\end{lemma}

\begin{proof}
For $0\le x\le1$, we have $x^n\le 1$ and so
\[
\sum_{n=0}^{\infty}\frac{x^n}{(n/m)!}\le \sum_{n=0}^{\infty}\frac{1}{(n/m)!}=\sum_{n=0}^{\infty}\frac{1}{\Gamma(n/m +1)}<\infty.
\]
Notice that 
$$(1+x)^m\exp(x^m)\ge 1.$$ 
For sufficiently large 
$$K_m>\sum_{n=0}^{\infty}\frac{1}{(n/m)!},$$
we have
\[
\sum_{n=0}^{\infty}\frac{x^n}{(n/m)!}\le K_m(1+x)^m\exp(x^m).
\]

For $x\ge 1$, decompose the sum according to $n=qm+j$, where $q\ge0$ and $0\le j\le m-1$. Then 
\[
(n/m)!=\Gamma\left(\frac{n}{m}+1\right)=\Gamma\left(\frac{qm+j}{m}+1\right)=\Gamma(q+j/m+1).
\]
By Stirling formula, when $x\to \infty$,
\[
\Gamma(x+1)\sim \sqrt{2\pi x}\left(\frac{x}{e} \right) ^x.
\]
Taking $x=q+a$, then 
$$\Gamma(q+a+1)\sim \sqrt{2\pi (q+a)}\left(\frac{q+a}{e} \right) ^{q+a}.$$
On the other hand, taking $x=q$, then $\Gamma(q+1)\sim \sqrt{2\pi q}\left(\frac{q}{e} \right) ^{q}$. We obtained the ratio 
\begin{equation}\label{eq:5}
\frac{\Gamma(q+a+1)}{\Gamma(q+1)}\sim \frac{\sqrt{2\pi (q+a)}\left(\frac{q+a}{e} \right) ^{q+a}}{\sqrt{2\pi q}\left(\frac{q}{e} \right) ^{q}}=\left( \frac{q+a}{q}\right)^{1/2}e^{-a}\frac{(q+a)^{q+a}}{q^q}.
\end{equation}
We continue to break down. By 
$$(q+a)^{q+a}=q^{q+a}(1+a/q)^{q+a},$$
we get
\begin{equation}\label{eq:6}
\frac{(q+a)^{q+a}}{q^q}=q^a\left(1+\frac{a}{q}\right)^{q+a}.
\end{equation}
Substituting \eqref{eq:6} into \eqref{eq:5},
\begin{equation}\label{eq:7}
\frac{\Gamma(q+a+1)}{\Gamma(q+1)}\sim q^a\left(1+\frac{a}{q}\right)^{q+a+1/2}e^{-a}.
\end{equation}
Since 
$$(1+a/q)^q \to e^a,$$
we have 
$$(1+a/q)^{q+a+1/2}\to e^a.$$ 
Applying this conclusion, transform \eqref{eq:7} into
\[
\frac{\Gamma(q+a+1)}{\Gamma(q+1)}\sim q^a.
\]
Let 
$$R_{j, q}:=\Gamma(q+j/m+1)/\Gamma(q+1).$$
Then $R_{j, q} \sim q^{j/m}$. Let's conduct a detailed analysis of $R_{j, q}$. If $j=0$, then 
$$R_{0, q}\to 1.$$
If $j> 0$, then 
$$R_{j, q} \sim q^{j/m}\to \infty.$$
So for each fixed value of $j$, $R_{j, q}$ cannot approach 0. Therefore, it has a positive lower bound of 
$$\inf_{q>0}R_{j, q}>0.$$ 
Since $j$ takes only finitely many values $0, \dots, m$, we choose the smallest one:
\[
c_m:=\min_{0\le j\le m-1}\inf_{q>0}R_{j, q}>0.
\]
Then
\begin{equation}\label{eq:8}
\Gamma(q+j/m+1)\ge c_m \G(q+1)=c_mq!
\end{equation}
for a constant $c_m>0$ depending only on $m$.
And so
\[
\sum_{n=0}^{\infty}\frac{x^n}{(n/m)!}\overset{\eqref{eq:8}}{\le}c_m^{-1}\sum_{n=0}^{\infty}\frac{x^n}{q!}=c_m^{-1}\sum_{q=0}^{\infty}\sum_{j=0}^{m-1}\frac{x^{qm+j}}{q!}
=c_m^{-1}\sum_{j=0}^{m-1}x^j\sum_{q=0}^{\infty}\frac{x^{qm}}{q!}.
\]
By 
$$\sum_{q=0}^{\infty}\frac{x^{qm}}{q!}=\exp(x^m),$$
we have
\[
\sum_{n=0}^{\infty}\frac{x^n}{(n/m)!}\le c_m^{-1}\Bigg( \sum_{j=0}^{m-1}x^j\Bigg) \exp(x^m).
\]
Hence, by 
$$\sum_{j=0}^{m-1}x^j\overset{x\ge 1}{\le} mx^{m-1},$$
we have
\[
\sum_{n=0}^{\infty}\frac{x^n}{(n/m)!}\le c_m^{-1}mx^{m-1}\exp(x^m)\le c_m^{-1}(1+x)^m\exp(x^m)=:K_m(1+x)^m\exp(x^m)
\]
which implies the first estimate. We take the logarithm of it, divide both sides by $x^m$, and then take the limit, which gives us 
\[
\limsup_{x\to\infty}x^{-m}\log\Bigg(\sum_{n=0}^{\infty}\frac{x^n}{(n/m)!}\Bigg)\le1.\qedhere
\]
\end{proof}

We adapt the development estimate from~\cite{BGS20}.  The only change needed for our purposes is that the denominator is normalized by the tensor-level operator bound of the chosen development, which is better suited to arbitrary reasonable tensor norms.  Let $\mathbf X$ be a $p$-rough path and let $\rho>0$. We write
\[
\delta_\rho(\mathbf X_t)=\bigl(1,\rho X_t^1,\ldots,\rho^{\lfloor p\rfloor}X_t^{\lfloor p\rfloor}\bigr)
\]
for the dilated path, and denote by $(\Gamma_t^\rho)_{0\le t\le T}$ the development of $\delta_\rho(\mathbf X_t)$ under $\Phi$.

\begin{proposition}\label{prop:2}
Let $\X_t=\exp(tl)$ be a pure $m$-rough path over $V$, and let $l_m:=\pi_m(l)$.  Suppose that $\Phi$ satisfies \eqref{eq:4}. Then
\begin{equation}\label{eq:11}
L_m(\X) \ge \frac{\sup\{\Ree\lambda:\lambda\in\spec(\Phi^{(m)}(l_m))\}}{C^m},
\end{equation}
where $\spec(A)$ denotes the set of eigenvalues of the finite-dimensional operator $A$.
\end{proposition}

\begin{proof}
Let $X^n$ denote the $n$-th signature level of $\X$.  We set $X^0_{0,1}=1$, and define $\Phi^{(0)}:\mathbb K\to \operatorname{End}(W)$ by $\Phi^{(0)}(c)=c\Id$. For the dilated path $\delta_\rho (\X)$, the developed endpoint is
\begin{equation*}
\G^\rho_1=\Phi\left(S\left(\delta_\rho (\X) \right)  \right)\overset{\eqref{eq:4-1}}{=}\Id+\sum_{n=1}^{\infty}\rho^n\Phi^{(n)}(X^n_{0,1})
=\sum_{n=0}^{\infty}\rho^n\Phi^{(n)}(X^n_{0,1}).
\end{equation*}
The series is convergent by the factorial decay estimate. By \eqref{eq:2}, we have
\[
\|\rho^n\Phi^{(n)}(X^n_{0,1})\|\le \rho^n\|\Phi^{(n)}\|\,\|X^n_{0,1}\|_n\overset{\eqref{eq:4}}{\le} \rho^nC^n\|X^n_{0,1}\|_n
\le \rho^nC^n\frac{\omega_{\mathbf{X}}(0,1)^{n/m}}{\beta_m (n/m)!}.
\]
That is, 
\[
\|\G^\rho_1\|\le \sum_{n=0}^{\infty}\|\rho^n\Phi^{(n)}(X^n_{0,1})\|\le \frac{1}{\beta_m }\sum_{n=0}^{\infty}\frac{\left(\rho C\omega_{\mathbf{X}}(0,1)^{1/m}\right)^n }{(n/m)!}
\overset{\eqref{eq:10+1}}{<} \infty.
\]
For convenience, we define 
$$\norm{X^n}_n:=\|X^n_{0, 1}\|_n.$$
For $N\ge1$, set
\begin{equation}\label{eq:13}
L_N:=\sup_{n\ge N}\left(\left(\frac n m\right)!\norm{X^n}_n\right)^{m/n}.
\end{equation}
By this definition, we have 
$$\left(\left(\frac n m\right)!\norm{X^n}_n\right)^{m/n}\le L_N,$$
and so
\begin{equation*}
\|X^n\|_n\le \frac{L_N^{n/m}}{(n/m)!}.
\end{equation*}
By this estimate, we obtain
\begin{equation}\label{eq:14}
\|\Phi^{(n)}(X^n)\|\le \|\Phi^{(n)}\|\,\|X^n\|\overset{\eqref{eq:4}}{\le}C^n\|X^n\|\le C^n\frac{L_N^{n/m}}{(n/m)!}.
\end{equation}
So
\begin{equation}\label{eq:15}
 \norm{\G^\rho_1}\le \sum_{n=0}^{N-1}\rho^n\norm{\Phi^{(n)}(X^n)}+\sum_{N}^{\infty}\rho^n\norm{\Phi^{(n)}(X^n)}\overset{\eqref{eq:14}}{\le}\sum_{n=0}^{N-1}\rho^n\norm{\Phi^{(n)}(X^n)}+\sum_{n=N}^{\infty}\frac{(\rho C)^n L_N^{n/m}}{(n/m)!}.
\end{equation}

We now estimate the exponential growth rate of $\norm{\Gamma^\rho_1}$ as $\rho\to\infty$.  Let 
$$P_N(\rho):=\sum_{n=0}^{N-1}\rho^n\norm{\Phi^{(n)}(X^n)}$$ 
and $x:=\rho CL_N^{1/m}$. Then 
\begin{equation}\label{eq:16}
\sum_{n=N}^{\infty}\frac{(\rho C)^n L_N^{n/m}}{(n/m)!}\le \sum_{n=0}^{\infty}\frac{(\rho CL_N^{1/m})^n }{(n/m)!}=\sum_{n=0}^{\infty}\frac{x^n }{(n/m)!}\overset{\eqref{eq:10+1}}{\le} K_m(1+x)^me^{x^m}.
\end{equation}
Substituting \eqref{eq:16} into \eqref{eq:15}, we have 
\[
\norm{\Gamma^\rho_1}\le P_N(\rho)+K_m(1+\rho CL_N^{1/m})^m\exp\left( (\rho C)^mL_N\right).
\]
Taking the logarithm of both sides of the above equation and dividing by $\rho^m C^m$, we get 
\[
\frac{\log\norm{\Gamma^\rho_1}}{\rho^m C^m}\le \frac{\log\left[ P_N(\rho)+K_m(1+\rho CL_N^{1/m})^m\exp\left( (\rho C)^mL_N\right)\right] }{\rho^m C^m}.
\]
Notice that 
\begin{align*}
&\ \lim_{\rho\to \infty}\frac{\log\left[ P_N(\rho)+K_m(1+\rho CL_N^{1/m})^m\exp\left( (\rho C)^mL_N\right)\right] }{\rho^m C^m}\\
=&\  \lim_{\rho\to \infty}\frac{\log\left[K_m(1+\rho CL_N^{1/m})^m\exp\left( (\rho C)^mL_N\right)\right] }{\rho^m C^m} 
\hspace{2cm}  \left(\text{by $\lim_{\rho\to \infty}\frac{\log P_N(\rho)}{\rho^m }=0$}\right)\\
=&\  \lim_{\rho\to \infty}\frac{\log\left[K_m(1+\rho CL_N^{1/m})^m\right]+(\rho C)^mL_N }{\rho^m C^m}\\
=&\  \lim_{\rho\to \infty}\frac{(\rho C)^mL_N }{\rho^m C^m}\hspace{2cm}  \Bigg(\text{by $\lim_{\rho\to \infty}\frac{\log\left[K_m(1+\rho CL_N^{1/m})^m\right]}{\rho^m }=0$}\Bigg)\\
=&\  L_N.
\end{align*}
Therefore
\[
 \limsup_{\rho\to\infty}\frac{\log\norm{\Gamma^\rho_1}}{\rho^mC^m}\le L_N.
\]
Since this estimate holds for every $N\ge1$, we obtain
\[
 \limsup_{\rho\to\infty}\frac{\log\norm{\Gamma^\rho_1}}{\rho^mC^m}\le \inf_{N\ge 1}L_N.
\]
Since 
\[
\inf_{N\ge 1}L_N\overset{\eqref{eq:13}}{=}\inf_{N\ge 1}\sup_{n\ge N}\left(\left(\frac n m\right)!\norm{X^n}_n\right)^{m/n}= \lim_{N\to \infty}\sup_{n\ge N}\left(\left(\frac n m\right)!\norm{X^n}_n\right)^{m/n}\overset{\eqref{eq:3}}{=}L_m(\X),
\]
we have
\begin{equation}\label{eq:17}
L_m(\X)=\inf_{N\ge 1}L_N\ge \limsup_{\rho\to\infty}\frac{\log\norm{\Gamma^\rho_1}}{\rho^mC^m}.
\end{equation}

On the other hand, by $\X_t=\exp(tl)$, we have
\[
\G^\rho_1=\Phi\left(S\left(\delta_\rho (\X) \right)  \right)=\Phi\left(\delta_\rho S\left(\X \right)  \right)=\Phi\left(\delta_\rho \exp(l)  \right).
\]
Since 
$$\exp(l)=\sum_{k=0}^{\infty}\frac{l^k}{k!}$$ 
and $\delta_\rho$ is an algebraic homomorphism, we obtain
\[
\Phi\left(\delta_\rho \exp(l)  \right)=\Phi\Bigg(\delta_\rho \Bigg(\sum_{k=0}^{\infty}\frac{l^k}{k!} \Bigg)   \Bigg)=\Phi\Bigg(\sum_{k=0}^{\infty}\frac{(\delta_\rho l)^k}{k!} \Bigg)  =\Phi\left( \exp(\delta_\rho l)  \right).
\]
Notice that 
$$l=l_1+\cdots+l_m,  \quad l_r\in \mathcal{L}_r(V)\subset V^{\otimes r}$$ 
and 
$$ \delta_\rho l=\sum_{r=1}^m\rho^rl_r.$$
Hence 
\[
\G^\rho_1=\Phi\Bigg( \exp\Bigg(\sum_{r=1}^m\rho^rl_r\Bigg)  \Bigg)=\exp\Bigg(\Phi \Bigg(\sum_{r=1}^m\rho^rl_r\Bigg)  \Bigg)=\exp\Bigg(\sum_{r=1}^m \rho^r\Phi^{(r)}(l_r)\Bigg).
\]
Let
\[
 T(\rho):=\Phi^{(m)}(l_m)+\rho^{-1}\Phi^{(m-1)}(l_{m-1})+\cdots+\rho^{-(m-1)}\Phi^{(1)}(l_1).
\]
For every eigenvalue $\lambda$ of $\Phi^{(m)}(l_m)$, finite dimensional perturbation theory gives a continuous eigenvalue branch $\lambda(\rho)$ of $T(\rho)$ with $\lambda(\rho)\to\lambda$. Then
\begin{equation}\label{eq:18}
\norm{\Gamma^\rho_1}=\norm{\exp(\rho^mT(\rho))}\ge |\exp\big(\rho^m\lambda(\rho)\big)|=\exp\big(\rho^m\Ree\lambda(\rho)\big).
\end{equation}
Substituting \eqref{eq:18} into \eqref{eq:17}, we have 
\[
L_m(\X)\ge \limsup_{\rho\to\infty}\frac{\rho^m\Ree\lambda(\rho)}{\rho^mC^m}=\limsup_{\rho\to\infty}\frac{\Ree\lambda(\rho)}{C^m}.
\]
Thus, for every $\lambda\in\spec(\Phi^{(m)}(l_m))$,
\[
L_m(\X)\ge \frac{\Ree\lambda}{C^m}.
\]
Taking the supremum over $\lambda\in\spec(\Phi^{(m)}(l_m))$ gives \eqref{eq:11}.
\end{proof}

\section{The norming cyclic development}\label{sec:cyclic}
The aim of this section is to prove the reverse inequality in the main theorem.  For this purpose we construct, for each $a\in V^{\otimes m}$, a finite-dimensional development $\Phi$ such that \eqref{eq:4} holds with $C=1$ and $\norm{a}_m$ is an eigenvalue of $\Phi^{(m)}(a)$. 

\subsection{Construction in the complex case}
Assume that the ground field $\K$ is $\C$ in this subsection.
The following elementary contraction proposition is the only point where the dual part of the definition of reasonable tensor norms is used. 

\begin{proposition}\label{prop:3}
Let $f\in (V^{\otimes m})^*$ with $\norm{f}\le1$.  For every $r\ge0$, define 
$$C_f^{m, r}:V^{\otimes(m+r)}\to V^{\otimes r}$$ 
on simple tensors by
\[
C_f^{m, r}(u\ot \eta):=f(u)\eta, \quad u\in V^{\otimes m},\quad \eta\in V^{\otimes r},
\]
and extend linearly.  Then
\[
\norm{C_f^{m, r}\zeta}_r\le \norm{\zeta}_{m+r}, \quad \zeta\in V^{\otimes(m+r)}.
\]
\end{proposition}

\begin{proof}
If $r=0$, then $$C_f^{m, r}:V^{\otimes m}\to \C.$$ 
For arbitrary $\zeta\in V^{\otimes m}$, we have 
$$\|C_f^{m, r}\zeta\|_0=|f(\zeta)|.$$ 
By $ \norm{f}\le1$, we have
\[
\frac{|f(\zeta)|}{ \|\zeta\|_m}\le \|f\|\le 1.
\]
Hence 
$$\|C_f^{m, r}\zeta\|_0=|f(\zeta)|\le  \|\zeta\|_m.$$

Let $r\ge1$.  For every $\psi\in(V^{\otimes r})^*$ with $\norm{\psi}\le1$, the dual reasonable norm property gives
\[
\norm{f\ot\psi}_{(V^{\otimes(m+r)})^*}\le \norm{f}_{(V^{\otimes m})^*}\norm{\psi}_{(V^{\otimes r})^*}\le 1.
\]
Therefore
\[
|\psi(C_f^{m, r}\zeta)|=|(f\ot\psi)(\zeta)|\le \norm{f\ot\psi}_{(V^{\otimes(m+r)})^*}\norm{\zeta}_{m+r}\le \norm{\zeta}_{m+r}.
\]
Taking the supremum over such $\psi$ gives
\[
\norm{C_f^{m, r}\zeta}_r=\sup_{\|\psi\|\le 1}|\psi(C_f^{m, r}\zeta)|\le \norm{\zeta}_{m+r}.\qedhere
\]
\end{proof}

We first dispose of the case $m=1$.  Let $a\in V$.  If $a=0$, there is
nothing to prove.  Otherwise choose $f\in V^*$ such that
\begin{equation*}
  \|f\|=1,\quad f(a)=\|a\|_1 .
\end{equation*}
Set $W=\mathbb C$ and define 
$$\Phi(v)z=f(v)z,\quad v\in V,\quad z\in\mathbb C.$$ 
Then $\Phi^{(1)}(a)$ has eigenvalue $\|a\|_1$.  Moreover, for every $n\ge1$,
\[
\Phi^{(n)}(v_1\otimes\cdots\otimes v_n)= f(v_1)\cdots f(v_n).
\]
By the dual reasonable norm property,
\[
\|f^{\otimes n}\|_{(V^{\otimes n})^*}\le \|f\|^n=1.
\]
Hence
\[
\|\Phi^{(n)}(u)\|\le \|u\|_n, \quad u\in V^{\otimes n},
\]
and therefore $\|\Phi^{(n)}\|\le1$.  Thus the required tensor-level control and the norming eigenvalue property hold in the case $m=1$.

We henceforth assume $m\ge2$.  Fix $a\in V^{\otimes m}$.  If $a=0$, no construction is needed.  Suppose $a\ne0$. By the Hahn--Banach theorem there exists a norming
functional \(f\in (V^{\otimes m})^*\) such that
\begin{equation}\label{eq:19}
\norm{f}=1, \quad f(a)=\norm{a}_m.
\end{equation}
This is obtained by multiplying a norming functional by a unimodular scalar. Define
\begin{equation*}
W:=\C\oplus V\oplus V^{\otimes 2}\oplus\cdots\oplus V^{\otimes(m-1)}
\end{equation*}
with the max norm
\begin{equation*}
\norm{(\xi_0,\xi_1,\ldots,\xi_{m-1})}_W:=\max_{0\le r\le m-1}\norm{\xi_r}_r.
\end{equation*}
For $v\in V$, define $A_v\in\End(W)$ by
\begin{equation*}
\begin{cases}
A_v\xi_0:=\xi_0v\in V,\\
A_v\xi_r:=v\ot\xi_r\in V^{\otimes(r+1)}, \quad 1\le r\le m-2, \\
A_v\xi_{m-1}:=f(v\ot\xi_{m-1})\,1\in\C.
\end{cases}
\end{equation*}
Put $\Phi(v):=A_v$.

\begin{proposition}\label{prop:4}
The map $\Phi:V\to\End(W)$ is linear and
\[
\norm{\Phi(v)}_{W\to W}\le \norm{v}_1, \quad v\in V.
\]
In particular 
$$\norm{\Phi}_{V\to\End(W)}\le1.$$
\end{proposition}

\begin{proof}
We first show that the map
\[
\Phi:V\longrightarrow \End(W),\quad v\mapsto A_v,
\]
is linear. Let 
$$\xi=(\xi_0,\ldots,\xi_{m-1})\in W.$$ 
By the definition of $A_v$, each component of $A_v\xi$ is linear in $v$. Hence, for all $u,v\in V$ and $\alpha,\beta\in\C$,
\[
A_{\alpha u+\beta v}\xi
=
\alpha A_u\xi+\beta A_v\xi .
\]
Since this holds for every $\xi\in W$, we have
\[
A_{\alpha u+\beta v}=\alpha A_u+\beta A_v.
\]
Therefore $\Phi$ is linear.
Notice that 
$$\norm{\xi}_W=\max_{0\le r\le m-1}\norm{\xi_r}_r.$$
For $0\le r\le m-2$, 
\[
\|A_v\xi_r\|_{r+1}=\|v\ot \xi_r\|_{r+1}\le \|v\|_{1}\|\xi_r\|_{r}.
\]
For $r=m-1$, 
\[
|A_v\xi_{m-1}|=|f(v\ot\xi_{m-1})|\le\|f\|\, \|v\ot\xi_{m-1}\|_{m}\overset{\|f\|=1}{\le}  \|v\|_{1}\|\xi_{m-1}\|_{m-1}.
\]
Hence 
$$\|A_v\xi\|_{W}\le  \|v\|_{1}\|\xi\|_{W},$$
and so
\[
\norm{\Phi(v)}_{W\to W}=\norm{A_v}_{W\to W}=\sup_{\|\xi\|_{W}\ne 0}\frac{\|A_v\xi\|_{W}}{\|\xi\|_{W}}\le \|v\|_{1}.
\]
In particular, 
\[
\norm{\Phi}_{V\to\End(W)}=\sup_{\|v\|_{1}\ne 0}\frac{\norm{\Phi(v)}_{W\to W}}{\|v\|_{1}}\le 1. \qedhere
\]
\end{proof}

The next result is the key point for arbitrary reasonable tensor norms.

\begin{theorem}\label{thm:2}
For every $n\ge1$, define the linear map
\[
\Phi^{(n)}:V^{\otimes n}\to\End(W), \quad v_1\otimes \cdots \otimes v_n \mapsto A_{v_1}\cdots A_{v_n}.
\]
Then the map $\Phi^{(n)}$ satisfies
\[
\norm{\Phi^{(n)}}_{V^{\otimes n}\to\End(W)}\le1.
\]
Moreover, the extensions are multiplicative in the sense that
\[
\Phi^{(n+k)}(u\otimes z)=\Phi^{(n)}(u)\Phi^{(k)}(z), \quad u\in V^{\otimes n},\ z\in V^{\otimes k}.
\]
\end{theorem}

\begin{proof}
The algebraic extension is forced by multilinearity.  For a simple tensor $u=v_1\otimes\cdots\otimes v_n$, set
\[
\Phi^{(n)}(u):=A_{v_1}\cdots A_{v_n},
\]
and extend linearly. The displayed multiplicativity is immediate on simple tensors and hence follows by linearity.

It remains to prove the norm estimate.  Fix $0\le r\le m-1$ and restrict the operator $\Phi^{(n)}(u)$ to the homogeneous component $V^{\otimes r}$ of $W$. Put
\[
  s\equiv n+r\pmod m, \quad 0\le s\le m-1 .
\]
So there exists an integer $q$ such that $n+r=qm+s$. If $q=0$, then $n+r=s<m$. By the definition of $\Phi^{(n)}$, for any $u\in V^{\otimes n}$ and $\xi\in V^{\otimes r}$, 
$$\Phi^{(n)}(u)\xi=u\otimes \xi.$$
Therefore 
\[
\|\Phi^{(n)}(u)\xi\|_s= \|u\otimes \xi\|_{n+r}\le \norm{u}_n\norm{\xi}_r.
\]
Suppose $q\ne 0$. Let 
$$P_{n+r,s}:V^{\otimes(n+r)}\to V^{\otimes(n+r)}$$ 
be the permutation operator which moves the first $s$ tensor factors to the end. Thus, on pure tensors,
\[
P_{n+r,s}(w_1\otimes\cdots\otimes w_{n+r})=w_{s+1}\otimes\cdots\otimes w_{n+r}\otimes w_1\otimes\cdots\otimes w_s .
\]
By Definition~\ref{def:1}, $P_{n+r,s}$ is an isometry. Define 
$$B_{n, r}:V^{\otimes(n+r)}\to V^{\otimes s}$$ 
by
\[
B_{n, r}(h):=C_f^{m, s}\circ C_f^{m, m+s}\circ \cdots \circ C_f^{m, (q-1)m+s}\circ P_{n+r, s}(h).
\]
For arbitrary simple tensor $u=u_1\otimes\cdots\otimes u_n$ and $\xi\in V^{\otimes r}$, we obtain 
\[
\Phi^{(n)}(u)\xi=f(u_{s+1}\otimes\cdots \otimes u_{s+m})\cdots f(u_{n-(m-r-1)}\otimes\cdots \otimes u_{n}\otimes \xi)u_1\otimes \cdots \otimes u_s =B_{n, r}(u\otimes \xi).
\]
Thus, for any $u\in V^{\otimes n},\ \xi\in V^{\otimes r}$, the operator $\Phi^{(n)}(u)$ is represented by $\xi\mapsto B_{n,r}(u\otimes \xi)$. Hence  
\begin{align*}
\|\Phi^{(n)}(u)\xi\|_s=&\  \|B_{n, r}(u\otimes \xi)\|_s\\
=&\   \|C_f^{m, s}\circ C_f^{m, m+s}\circ \cdots \circ C_f^{m, (q-1)m+s}\circ P_{n+r, s}(u\otimes \xi)\|_s\\
\le&\   \|C_f^{m, s} \|\, \|C_f^{m, m+s} \|\cdots \|C_f^{m, (q-1)m+s} \|\,\|P_{n+r, s}(u\otimes \xi)\|_{n+r}\\
=&\   \|C_f^{m, s} \|\, \|C_f^{m, m+s} \|\cdots \|C_f^{m, (q-1)m+s} \|\,\|u\otimes \xi\|_{n+r}\\
\le&\  \|u\otimes \xi\|_{n+r} \hspace{1cm}  (\text{By Proposition~\ref{prop:3}})\\
\le&\  \norm{u}_n\norm{\xi}_r.
\end{align*}
Let
\[
\xi=(\xi_0,\ldots,\xi_{m-1})\in W .
\]
For each $r$, the action of $\Phi^{(n)}(u)$ on the component $\xi_r\in V^{\otimes r}$ is bounded by $\norm{u}_n\norm{\xi_r}_r$. Therefore, using the max norm on $W$, we obtain
\[
\norm{\Phi^{(n)}(u)\xi}_W=\max_{0\le r\le m-1}\norm{\Phi^{(n)}(u)\xi_r}_{s(r)}\le \norm{u}_n \max_{0\le r\le m-1}\norm{\xi_r}_r =\norm{u}_n\norm{\xi}_W .
\]
Taking the supremum over all $\xi\in W$ with $\norm{\xi}_W\le1$ gives
\[
\norm{\Phi^{(n)}(u)}_{W\to W}=\sup_{\norm{\xi}_W\ne 0}\frac{\norm{\Phi^{(n)}(u)\xi}_W}{\norm{\xi}_W}\le \norm{u}_n. 
\]
Hence
\[
\norm{\Phi^{(n)}}_{V^{\otimes n}\to\End(W)}=\sup_{\norm{u}_n\ne 0}\frac{\norm{\Phi^{(n)}(u)}_{W\to W}}{\norm{u}_n}\le 1. 
\]
This proves the desired bound.
\end{proof}

\begin{proposition}\label{prop:5}
Let $a\in V^{\otimes m}$ and let $\Phi$ be the cyclic development constructed from a norming functional $f$ satisfying \eqref{eq:19}.  Then
\[
\Phi^{(m)}(a)1=\norm{a}_m\,1.
\]
In particular, $$\norm{a}_m\in\spec(\Phi^{(m)}(a)).$$
\end{proposition}

\begin{proof}
For the pure tensor $v_1\ot\cdots\ot v_m$, we have
\[
\Phi^{(m)}(v_1\ot\cdots\ot v_m)=A_{v_1}\cdots A_{v_m}.
\]
Since operators act from right to left,
\[
A_{v_1}\cdots A_{v_m}1=f(v_1\ot\cdots\ot v_m)1.
\]
Linearity gives $\Phi^{(m)}(u)1=f(u)1$ for every $u\in V^{\otimes m}$.  Taking $u=a$ and using \eqref{eq:19} proves the claim.
\end{proof}

\begin{remark}  
The preceding construction was written over $\C$ .  If $\K=\R$, choose a real norming functional $f\in(V^{\otimes m})^*$ with $\norm{f}=1$ and $f(a)=\norm{a}_m$.  The same construction gives a real finite dimensional Banach space
\[
W_\R:=\R\oplus V\oplus\cdots\oplus V^{\otimes(m-1)}
\]
and real operators $A_v$ with the same tensor-level bounds and the eigenvector identity $\Phi^{(m)}(a)1=\norm{a}_m1$. After complexifying \(W_\R\) and extending the operators complex linearly, we obtain a complex development with the same tensor-level bound \(C=1\).  Therefore the argument applies equally in the real case.
\end{remark}

\subsection{Proof of the main theorem}
This subsection devotes to the proof of the main result.

\begin{theorem}\label{thm:3}
Let $V$ be a finite dimensional Banach space over $\R$ or $\C$, and equip the algebraic tensor powers $V^{\otimes n}$ with a sequence of reasonable tensor algebra norms $\norm{\cdot}_n$.  Let
\[
\X_t=\exp(tl),\quad 0\le t\le 1,
\]
be a pure $m$-rough path with $l\in \mathcal  L^{(m)}(V)$.  Then
\[
L_m(\X)=\norm{\pi_m(l)}_m.
\]
Equivalently,
\[
\limsup_{n\to\infty}\left(\left(\frac n m\right)! \norm{\pi_n(\exp(l))}_n\right)^{m/n}=\norm{\pi_m(l)}_m.
\]
\end{theorem}

\begin{proof}
Write
\[
l=l_1+l_2+\cdots+l_m, \quad l_r\in \mathcal{L}_r(V)\subset V^{\otimes r}.
\]
If $l_m=0$, the upper bound in Lemma~\ref{prop:1} gives
\[
0\le L_m(\X)\le \norm{l_m}_m=0,
\]
and so the result follows.

Assume $l_m\ne0$. The case $m=1$ is immediate. By the crossnorm consequence of Definition~\ref{def:1}, we have 
$$\|l_1^{\otimes n}\|_n=\|l_1\|_1^n.$$
Since
\[
\pi_n(S(\mathbf X))=\frac{l_1^{\otimes n}}{n!},
\]
we obtain 
$$L_1(\mathbf X)=\|l_1\|_1.$$
For $m\ge2$, construct the cyclic development $\Phi$ associated with a norming functional of $l_m$.  By Theorem~\ref{thm:2}, it satisfies the tensor-level bound \eqref{eq:4} with $C=1$. By Proposition~\ref{prop:5},
\[
\norm{l_m}_m\in\spec(\Phi^{(m)}(l_m)).
\]
Since this eigenvalue is real and positive, Proposition~\ref{prop:2} yields
\[
L_m(\X)\ge \norm{l_m}_m.
\]
Together with the upper bound Lemma~\ref{prop:1}, this gives
\[
  L_m(\X)=\norm{l_m}_m=\norm{\pi_m(l)}_m.
\]
Finally, for a pure rough path $\X_t=\exp(tl)$, the signature is $S(\X)=\exp(l)$. Therefore the displayed algebraic formulation follows immediately.
\end{proof}

\vskip 0.2in

\noindent
{\bf Acknowledgments.} This work is supported by the National Natural Science Foundation of China (12571019), the Natural Science Foundation of Gansu Province (25JRRA644) and Innovative Fundamental Research Group Project of Gansu Province (23JRRA684). 
We thank Xi Geng for illuminating discussions.

\noindent
{\bf Declaration of interests. } The authors have no conflicts of interest to disclose.

\noindent
{\bf Data availability. } Data sharing is not applicable as no new data were created or analyzed.

\end{document}